\documentclass[12pt,reqno]{amsart}
\usepackage{amsthm}
\usepackage{fullpage,url}
\usepackage{pstricks,pst-plot,pst-node}
 \usepackage{amscd}   
\usepackage[all,graph,poly]{xy} 
\usepackage{amssymb,amsmath}
\usepackage{multicol}
\usepackage{setspace}
\DeclareFontEncoding{OT2}{}{} 

\newcommand{\abs}[1]{\lvert#1\rvert}

\newcommand{\cH}{\mathcal{H}}
\newcommand{\cL}{\mathcal{L}}

\newcommand{\cO}{\mathcal{O}}
\newcommand{\cP}{\mathcal{P}}

\newcommand{\FF}{\mathbb{F}}

\newcommand{\ip}[2]{\langle #1 \vert #2 \rangle}

\newcommand{\op}[1]{\operatorname{#1}}
\newcommand{\PP}{\mathbb{P}}

\newcommand{\QQ}{\mathbb{Q}}

\newcommand{\RR}{\mathbb{R}}

\newcommand{\ZZ}{\mathbb{Z}}

\DeclareMathOperator{\Tr}{Tr}
\providecommand{\abs}[1]{\lvert#1\rvert}
\newcommand{\arhh}[1]{\ar@{->>}[#1]}
\newcommand{\e}[1]{\ar@{-}[#1]}
\newcommand{\ed}[1]{\ar@{--}[#1]}
\newcommand{\ee}[1]{\ar@{=}[#1]}
\newcommand{\er}[1]{\ar@{-}[#1] |-{\SelectTips{cm}{}\object@{<}} |{\SelectTips{eu}{}\object@{}}}
\newcommand{\el}[1]{\ar@{-}[#1] |-{\SelectTips{cm}{}\object@{>}} |{\SelectTips{eu}{}\object@{}}}
\newcommand{\edr}[1]{\ar@{--}[#1] |-{\SelectTips{cm}{}\object@{<}} |{\SelectTips{eu}{}\object@{}}}
\newcommand{\edl}[1]{\ar@{--}[#1] |-{\SelectTips{cm}{}\object@{>}} |{\SelectTips{eu}{}\object@{}}}

 \newcommand{\lul}[1]{\ar@{}[l]_<<{#1}}
\newcommand{\rrul}[1]{\ar@{}[r]^<<<<{#1}}
\newcommand{\rul}[1]{\ar@{}[r]^<<{#1}}
\newcommand{\ldl}[1]{\ar@{}[l]^<<{#1}}
\newcommand{\rdl}[1]{\ar@{}[r]_<<{#1}}
\newcommand{\dl}[1]{\ar@{}[d]_<<{#1}}
\newcommand{\dll}[1]{\ar@{}[dd]_{#1}}
\swapnumbers

\newtheorem{theorem}{Theorem}[section]
\newtheorem{lemma}[theorem]{Lemma}
\newtheorem{corollary}[theorem]{Corollary}
\newtheorem{proposition}[theorem]{Proposition}

\theoremstyle{definition}
\newtheorem{definition}[theorem]{Definition}

\newtheorem{example}[theorem]{Example}

\newtheorem{topic}[theorem]{}

\theoremstyle{remark}
\newtheorem{remark}[theorem]{Remark}
\newtheoremstyle{head}
{}
{}
{\bfseries}
{}
{}
{}
{.5em}
{}

\theoremstyle{head}

\begin{document}
%
%
%
%
\title[P2Fq]{Modular lattices from finite projective planes}
\author{Tathagata Basak}
\address{Department of Mathematics\\Iowa State University \\Ames, IA 50011}
\email{tathagat@iastate.edu}
\urladdr{http://www.iastate.edu/\textasciitilde tathagat}
%
%
%
%
%
%
%
\date{March 2, 2012}
\begin{abstract} 
Using the geometry of the projective plane over the finite field $\FF_q$,
we construct a Hermitian Lorentzian lattice $L_{q}$ of 
dimension $(q^{2} + q + 2)$ defined over a certain number ring $\cO$ that depends on $q$. 
We show that infinitely many of these lattices are $p$-modular, that is, $p L'_{q} = L_{q}$,
where $p$ is some prime in $\cO$ such that $\abs{p}^{2} =q$.
The reflection group of the Lorentzian lattice obtained for $q = 3$ seems to
be closely related to the monster simple group via the presentation of the bimonster
as a quotient of the Coxeter group on the incidence graph of $\PP^2(\FF_3)$.
\par
The Lorentzian lattices $L_q$ sometimes lead to construction of interesting positive definite lattices.
In particular, if $q \equiv 3 \bmod 4$ is a rational prime such that $(q^2 + q + 1)$ is norm of some
element in $\QQ[\sqrt{-q}]$, then we find a $2q(q+1)$ dimensional even
unimodular positive definite integer lattice $M_{q}$ such that $\op{Aut}(M_q) \supseteq \op{PGL}(3,\FF_q)$.
We find that $M_3$ is the Leech lattice.
\end{abstract}
\maketitle
%
%
%
%
\section{Introduction}
\begin{topic}{\bf The results:}
\label{intro1}
Let $q$ be a rational prime power and $n = (q^2 + q + 1)$.
Let $\cO$ be either the ring of rational integers or the ring of integers in a quadratic
imaginary number field or Hurwitz's ring of integral quaternions.
Let $p \in \cO$ be a prime element such that
$\abs{p}^2 = q$, and $\bar{z} = z \bmod p \cO$ for all $z \in \cO$. 
Given such a triple $(\cO, p, q)$, we shall construct a Hermitian $\cO$-lattice $L_q$ 
of signature $(1,n)$ such that $\op{PGL}(3,\FF_q)$ acts naturally on $L_q$ and
$L_q \subseteq p L_q'$.
If $q$ happens to be a rational prime, then we show that $L_{q}$ is $p$-modular,
that is $L_q = p L'_q$ (see \ref{th-pmodular}).
\par
If $L_{q}$ contains a norm zero vector fixed by $\op{PGL}(3, \FF_{q})$, then
we can split $L_q$ as direct sum of a definite lattice $\Lambda_q$
and a hyperbolic cell, where $\Lambda_{q}$ is stable under $\op{PGL}(3,\FF_{q})$ action.
If $p = \sqrt{-3}$, $q = 3$ and $\cO = \ZZ[\tfrac{1 + p}{2}]$, then $\Lambda_q$
is a form of Leech lattice defined over $\cO$.
We show that if $q \equiv 3 \bmod 4$ is a rational prime and $n$ is
norm of some element in $\QQ[\sqrt{-q}]$, then $\Lambda_q$ is $p$-modular 
and $\op{Aut}(\Lambda_q) \supseteq \op{PGL}(3, \FF_q)$ (see \ref{posmod}, \ref{examples2}).
An appropriately scaled real form of $\Lambda_q$
gives us a positive definite even unimodular $2q(q+1)$ dimensional  $\ZZ$-lattice $M_q$
such that $\op{Aut}(M_q) \supseteq \op{PGL}(3, \FF_q)$ (see \ref{unimod}).
Such examples exists for $q = 3, 47, 59, 71, 131, \dotsb$. General conjectures in analytic
number theory suggest that there are infinitely many such primes.
\end{topic}
\begin{topic}{\bf Examples:} 
\label{examples1}
\begin{enumerate}
\item Let $q$ to be a rational prime; $q \equiv 3 \bmod 4$. Let $p = \sqrt{-q}$
and $\cO = \ZZ[\frac{1+ p}{2}]$. Then the assumptions in \ref{intro1} is satisfied.
So we get infinitely many $p$-modular Hermitian lattices $L_q$.
\item Among the lattices in (1), the lattice $L_3$ obtained for $q = 3$ seems to be especially interesting.
The reflection group of $L_{3}$ gives us a complex hyperbolic reflection
group in $PU(1,13)$ having finite co-volume. Thirteen is the largest dimension in which 
such an finite co-volume discrete reflection group in $PU(1,n)$ is known. 
The lattice $L_3$ and its construction given here plays a major role in an ongoing project
(see \cite{DA:Leech}, \cite{DA:Monstrous}, \cite{TB:EL}, \cite{TB:EL2}) trying to relate the complex reflection group of $L_3$
and the monster via the Conway-Ivanov-Norton presentation of the bimonster
(see \cite{CNS:B}, \cite{CCS:26}, \cite{AAI:Geom}, \cite{AAI:Y}).
The construction described here came up while studying this example. 
\item Our construction also goes through
 if $\cO$ is a ring of Hurwitz's quaternionic integers and $p = (1 - i)$. The lattice obtained in this
case is a direct sum of a quaternionic form of the Leech lattice and a hyperbolic cell.
The reflection group of this lattice has properties analogous to the reflection group of 
the lattice $L_3$ mentioned in (2); see \cite{TB:QL}.
\end{enumerate}
\end{topic}
\begin{topic}{\bf Remarks on the construction: }
\begin{itemize}
\item Suppose $z$ is a  primitive vector of norm $0$ in $L_q$ fixed by $\op{PGL}(3,\FF_q)$
or some large subgroup of $\op{PGL}(3,\FF_q)$.
The definite lattices $z^{\bot}/z$ are sometimes interesting. In the examples  (2) and (3) of \ref{examples1}
this yields the complex and quaternionic form of the Leech lattice.
\item The definition of the lattices $L_q$ given in
\ref{def-Lq} is similar to the definition of a root lattice. In this analogy,
the incidence graph of $\PP^2(\FF_q)$ plays the role of a Dynkin diagram. 
This analogy has proved to be an useful one in understanding the
reflection group of the lattice $L_3$ mentioned in \ref{examples1}(2),
and its connection with the monster (see. \cite{TB:EL}).
\item Nice lattices are often constructed using nice error correcting codes.
For example see \cite{CS:SPLAG}, pp. 197-198 and pp. 211-212.
One can view our construction in this spirit, with the code being 
given by the incidence matrix of the points and lines of $\PP^2(\FF_q)$.
\item Bacher and Venkov, in \cite{BV:LA}, constructed a $28$ dimensional integer lattices of minimal norm
 $3$ whose shortest vectors are parametrized by the Lagrangian subspaces in
 $6$ dimensional symplectic vector space over $\FF_3$. This example also 
 seems to be related to our construction (Boris Venkov, private communications).
 \item Alexey Bondal pointed out to me that the construction in this paper
 bears similarity with his method of construction of lattices in simple Lie algebras
 which are invariant under the automorphisms that preserve
 a decomposition of the Lie algebra into mutually orthogonal Cartan subalgebras.
 (for example, see  \cite{AB:IL} or \cite{KP:OD}, ch. 9).
\item The lattices that satisfy $p L' = L$ are called $p$-modular. These behave
much like unimodular lattices, for example, see \cite{NS:SE}.
Appropriately scaled real form of the lattices described in (2) and (3) of \ref{examples1}
are the even unimodular lattices $II_{2,26}$ and $II_{4,28}$ respectively.
\item The construction given in \ref{def-Lq} probably yields more examples of Hermitian lattices defined over other
rings $\cO$, for example, certain maximal orders in rational quaternion algebras.
But for simplicity of presentation we shall restrict ourselves to $\cO$ being a ring as in \ref{intro1}.
\end{itemize}
\end{topic}
\begin{topic}{\bf Plan:} We describe the construction of the lattices $L_q$ in section \ref{s-Lorentzian}.
In section \ref{s-definite}, we describe the $p$-modular definite Hermitian $\cO$-lattices
and the positive definite unimodular $\ZZ$-lattices having $\op{PGL}(3, \FF_q)$ symmetry obtained from $L_q$.
Section \ref{s-inc} contains quick proofs of some known facts regarding the structure
of incidence matrices of finite projective planes that are relevant for us.
\end{topic}
%
%
\section{Lorentzian lattices with symmetries of finite projective planes}
\label{s-Lorentzian}
\begin{definition}[Hermitian lattices]
Let $\cO$ be a ring as in \ref{intro1}. Let  $\op{Frac}(\cO)$ be its fraction field.
Let $L$ be a projective $\cO$-module of rank $n$ with 
a $\cO$-valued Hermitian form $\ip{\;}{\;} : L \times L \to \cO$. We shall always
assume that the Hermitian form is linear in the second variable.
The {\it dual} module of $L$, denoted $L'$, is the set of all $\cO$-valued linear functionals
on $L$. 
The Hermitian form induces a natural map $L \to L'$ given by $x \mapsto \ip{\;}{x}$.
The kernel of this map is called the {\it radical} of $L$, and is denoted by $\op{Rad}(L)$.
We say $(L, \ip{\;}{\;})$ is {\it non-singular} if $\op{Rad}(L) = 0$. 
If $(L, \ip{\;}{\;})$ is non-singular, then we say $(L, \ip{\;}{\;})$ is an {\it $\cO$-lattice}
of rank $n$. If $\op{Rad}(L) \neq 0$, we say $(L, \ip{\;}{\;})$ is a singular $\cO$-lattice. 
\par
We shall denote a lattice $(L, \ip{\;}{\;})$ simply by $L$. The dual of a lattice is a lattice.
We may identify a lattice inside its dual using the Hermitian form.
One says that $L$ is {\it unimodular} if $L' = L$. One says $L$ is {\it $p$-modular}
for some $p \in \cO$, if $L' = p^{-1} L$.
A lattice $L$ has signature $(m,k)$ if $L \otimes_{\cO} \op{Frac}(\cO)$
has a basis whose matrix of 
inner products have $m$ positive eigenvalues and $k$ negative eigenvalues.
A lattice is {\it Lorentzian} if it has signature $(1,k)$. 
\end{definition}
\begin{example}
Let $q$ be a rational prime, $p = \sqrt{-q}$ and $\cO$ be the ring of integers in $\QQ[\sqrt{-q}]$. 
Let $\cO^{1,k}$ be the free $\cO$-module of rank $(k+1)$ with the Hermitian form
$$\ip{(u_0,u_1, \dotsb, u_k)}{(v_0, v_1, \dotsb, v_k)} = \bar{u}_0  v_0 - \bar{u}_1 v_1 - \dotsb - \bar{u}_k v_k.$$
Then $\cO^{1,k}$ is unimodular while $p \cO^{1,k}$ is $q$-modular.
\end{example}
\begin{definition} Let $(\cO, p,q)$ be as in \ref{intro1}.
Let $\PP^2(\FF_q)$ be the projective plane over $\FF_q$. Let 
\begin{equation*}
n = q^2 + q + 1.
\end{equation*}
Let $\cP$ be the set of points and $\cL$ be the set of lines of $\PP^2(\FF_q)$.
The sets $\cP$ and $\cL$ have $n$ elements each.
If a point $x \in \cP$ is incident on a line $l \in \cL$, then we write $x \in l$.
Let $D$ be the (directed) incidence graph of $\PP^2(\FF_q)$.
The vertex set of $D$ is $\cP \cup \cL$. There is a directed edge in $D$ from 
a vertex $l$ to a vertex $x$, if $x \in \cP$, $l \in \cL$ and $x \in l$.
\par
Let $L^{\circ}_q$ be the free $\cO$--module of rank $2 n$ with basis vectors
indexed by $D =\cP \cup \cL$. Let $r, s \in D$.
Define a Hermitian form $\ip{\;}{\;} : L^{\circ}_{q} \times L^{\circ}_{q} \to \cO$ by
\begin{equation}
\ip{r}{s} = \begin{cases} 
-q & \text{\;if \;\;} r = s \in D, \\
p & \text{\; if \;\;} r \in \cP, s \in \cL, r \in s, \\
\bar{p} & \text{\; if \;\;} r \in \cL, s \in \cP, s \in r, \\
0 & \text{otherwise.} 
\end{cases}
\label{ipd}
\end{equation}
\end{definition}
\begin{lemma}
Given $l \in \cL$, let $w_l  = \bar{p} l + \sum_{ x \in l} x $. Then
$\ip{x'}{w_l} = 0$ and $\ip{w_l}{l'} = p$ for all $x' \in \cP$ and
$ l' \in \cL$.
\label{lemma-wl}
\end{lemma}
\begin{proof}
The inner products are easily calculated from \eqref{ipd}. For example
\begin{equation*}
\ip{w_l}{l'} =
\begin{cases}
p.0 + p & \text{\; if \;} l \neq l', \text{\; since $l$ and $l'$ meet at one point,} \\
p.(-q)  + (q+ 1)p  & \text{\; if \;} l = l',\text{\; since every line has $(q+1)$ points on it.} \\
\end{cases} 
\end{equation*}
\end{proof}
\begin{proposition}
The radical of $L^{0}_{q}$ has rank $(n-1)$ and $\op{Rad}(L^0_q) \otimes_{\cO} \op{Frac}(\cO)$ is
spanned by the vectors $(w_{l_1} - w_{l_2})$ for $l_1, l_2$ in $\cL$.  
\label{p-rad}
\end{proposition}
\begin{proof}
If $l_1$, $l_2$ belong to $\cL$, then
lemma \ref{lemma-wl} implies that $(w_{l_1} - w_{l_2}) \in \op{Rad}(L^{\circ}_{q})$.
Let $U$ be the $\cO$-module of rank $(n-1)$ spanned by $(w_{l_1} - w_{l_2})$ for all $l_1, l_2 \in \cL$.
Then $L^{\circ}_q/U$ is a $\cO$--module of rank $(n+1)$.
Since $U \subseteq \op{Rad}(L^{\circ}_{q})$, the Hermitian form on $L^{\circ}_q$
descends to a Hermitian form on $L^{\circ}_q/U$, which is again denoted by $\ip{\;}{\;}$.
For each $l \in \cL$, the vectors $w_l$ have the same image in $L^{\circ}_q/U$;
call this image $w_{\cP}$.
Lemma \ref{lemma-wl} implies that $w_{\cP}$ is orthogonal (in $L^{\circ}_q/U$)
to each $x \in \cP$ and $\ip{w_{\cP}}{ l } = p$ for all $l \in \cL$. So
\begin{equation*}
 \ip{w_{\cP}}{w_{\cP}} = \ip{\bar{p} l + \sum_{x \in l} x }{w_{\cP}} = p \ip{l}{ w_{\cP}} = q.
\end{equation*}
The matrix of inner products of the $(n + 1)$ vectors  
$\cP \cup \lbrace w_{\cP} \rbrace$ in $L^{\circ}_q/U$ is a diagonal matrix with diagonal entries 
$( -q, -q, \dotsb, -q, q)$. So $L^{\circ}_q/U$ is a non-degenerate $ \cO$--module of signature
$(1,n)$. It follows that $U\otimes_{\cO} \op{Frac}(\cO) = \op{Rad}(L^{\circ}_{q}) \otimes_{\cO} \op{Frac}(\cO) $.
\end{proof}
\begin{definition}
\label{def-Lq}
Let $L_q$ be the quotient of $L_q^{\circ}$ by its radical. 
If $a \in \cO$ and $v \in L_q^0$ such that 
$a v \in \op{Rad}(L_q^0)$, then $v \in \op{Rad}(L_q^0)$. So 
$L_q$ is torsion free and it is obviously finitely generated.
So if $\cO$ is a Dedekind domain, then $L_q$ is is projective. 
Over Hurwitz's integral quaternions $\cH$, a finitely generated torsion-free module
is free, since $\cH$ has division with remainders.
\par
The proof of \ref{p-rad} shows that the Hermitian form on $L^{\circ}_q$
induces a non-degenerate Hermitian form on $L_q$ of signature $(1,n)$.
The basis vectors of $L^{\circ}_{q}$ defines $2n$ vectors in $L_q$ indexed by the points and lines of $\PP^2(\FF_q)$.
These will be denoted by $x_0, \dotsb, x_{n-1}$ and $l_0, \dotsb, l_{n-1}$ respectively.
\par
We have two more distinguished vectors $w_{\cP}$ and $w_{\cL}$ in $L_q$. We already defined $w_{\cP}$ above.
For $x \in \cP$,  let $w_{x} = p x + \sum_{x \in l} l$.  As above, one checks that $(w_x - w_{x'}) \in \op{Rad}(L^{\circ}_{q})$
for $x, x' \in \cP$. We let $w_{\cL}$ be the image of the vectors $w_x$ in $L_q$. So
\begin{equation}
w_{\cP}  = \bar{p} l + \sum_{ x' \in l} x'  \;\;\text{and} \;\;\;\;
w_{\cL}  = p x + \sum_{ x \in l'} l', 
\label{defwpwl}
\end{equation}
for any $x \in \cP$ and $l \in \cL$.
Using \eqref{ipd} one checks that for all $x \in \cP$ and $l \in \cL$,
we have,  
\begin{equation}
 \ip{w_{\cP}}{x} = \ip{w_{\cL}}{l} = 0, \text{\;\;} 
 \ip{w_{\cP}}{l} = \ip{x}{w_{\cL}} = p, \text{\;\;}
 \abs{w_{\cP}}^2 = \abs{w_{\cL}}^2 = q, \text{\;\;} 
 \ip{w_{\cP}}{ w_{\cL}} = (q+1)p.
\label{eq-ipwpwl}
\end{equation}
\end{definition}
\begin{topic}{\bf Line Coordinates on $L_q$: }
\label{t-linec}
Let $L_{\cL}$ be the sublattice of $L_q$ spanned by $\lbrace w_{\cL}, l_0, \dotsb, l_{n-1} \rbrace$
and $L_{\cP}$ be the sublattice of $L_q$ spanned by $\lbrace w_{\cP}, x_0, \dotsb, x_{n-1} \rbrace$.
Then $L_{q} = L_{\cP} + L_{\cL}$.
From \eqref{ipd} and \eqref{eq-ipwpwl}, we observe that 
 $$L_{\cP} \simeq L_{\cL} \simeq p \cO^{1,n}.$$
By looking at the inner products, we get the following inclusions: 
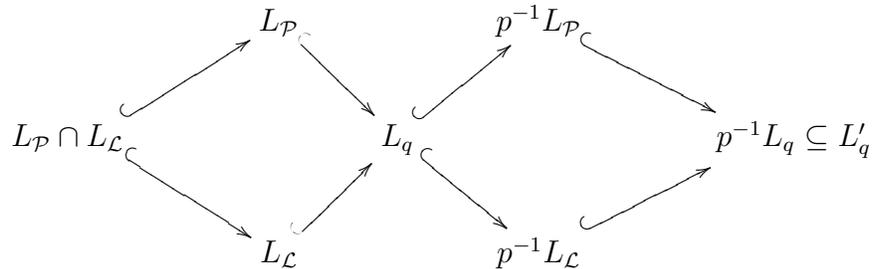
\begin{equation*} 
\xymatrix{ 
& L_{\cP} \ar@{^{(}->}[dr] & &  p^{-1} L_{\cP}  \ar@{^{(}->}[dr] &\\
L_{\cP} \cap L_{\cL}  \;\;\;\;\;\;  \ar@{^{(}->}[dr] \ar@{^{(}->}[ur] & &   L_{q}  \ar@{^{(}->}[dr] \ar@{^{(}->}[ur] & & \;\;\;\;\;\; p^{-1}L_{q} \subseteq L_{q}'\\
& L_{\cL}   \ar@{^{(}->}[ur] & & p^{-1}L_{\cL}  \ar@{^{(}->}[ur] &
}
 \end{equation*}
Let us identify $p^{-1} L_{\cL}$ with $\cO^{1,n}$. Then $L_{q}$ becomes a sub-lattice of the unimodular $\cO$-lattice
$\cO^{1,n}$. Following \cite{DA:Y555}, we call this the {\it line coordinates} on $L_{q}$. 
 In other words, the line coordinates for $v = p^{-1}(v_{\infty} w_{\cL} + v_0 l_0 + \dotsb + v_{n-1} l_{n-1})$ 
is $(v_{\infty} ; v_0, v_1, \dotsb, v_{n-1})$.
The following lemma will help us decide when $L_{q}$ is $p$-modular.
 \end{topic}
\begin{lemma} 
Let $(\cO,p)$ be as in \ref{intro1}.
Let $M$ be an unimodular Hermitian $\cO$-lattice. Then $W = M / p M$ is a $\cO/p \cO$-vector space.
Let $\pi: M \to W$ be the projection. Let $L$ be a sublattice of $M$. Let $X = \pi(L)$.
\par
(a) The Hermitian form on $M$ induces a non-degenerate symmetric bilinear form on the
$\cO/p \cO$-module $W$ given by $(\pi(x), \pi(y)) = \ip{x}{y} \bmod p \cO$. 
\par
(b) $X$ is isotropic if and only if $L \subseteq p L'$.
\par
(c) There is a bijection between $p$-modular lattices $L$ lying between $M$ and $pM$ and 
subspaces $X \subseteq W$ such that $X = X^{\bot}$, 
given by $X = \pi(L)$. Such subspaces $X$ are maximal isotropic.
\label{l-isot}
\end{lemma}
\begin{proof}
(a) Since $M$ is unimodular, the form on $W$ is non-degenerate.
 Since the form $\ip{\;}{\;}$ on $M$  is Hermitian and $\bar{z} \equiv z \bmod p\cO$ for all $z \in \cO$,
 the form on $W$ is symmetric. This proves (a). Part (b) is clear.
\par
(c) Since $pM \subseteq L \subseteq M$ and $M$ is unimodular, $L' \subseteq p^{-1}M$; so $p L' \subseteq M$.
Part (c) follows, once one verifies that $\pi^{-1} (X^{\bot}) = p L'$.
\end{proof}
\begin{theorem} Let $(\cO, p, q)$ be as in \ref{intro1}.
 Suppose $q$ is a rational prime and $\cO/p \cO \simeq \FF_q$.
Then the lattice $L_q$ is $p$-modular, that is, $p L_q'  = L_q$.
 \label{th-pmodular}
 \end{theorem}
 The proof uses the following theorem that appeared in the coding theory literature:
 \begin{theorem}[\cite{GM:DCC}]
 Let $q = l^d$ be a power of a rational prime $l$. Then
 the $\FF_q$-rank of the incidence matrix of $\PP^2(\FF_q)$ is equal to $\binom{l+1}{2}^d + 1$.
 \label{th-GM}
\end{theorem}
\begin{remark}
Theorem \ref{th-GM} appears in Theorem $2'$, page 1067, of \cite{GM:DCC}.
The proof of \ref{th-GM} uses the fact that {\it the incidence matrix of $\PP^2(\FF_q)$
generates  a ``cyclic difference set code"} (see \cite{SJ:FP}).  We have included a quick explanation of this fact in section \ref{s-inc}.
\end{remark}
  \begin{proof}[proof of theorem \ref{th-pmodular}]
Identify $L_q$ inside the unimodular lattice $p^{-1} L_{\cL} = \cO^{1,n} = M$ using the line coordinates. Then
equation \eqref{defwpwl} implies that $x_i = (1 ; \epsilon_1, \dotsb, \epsilon_n)$, where
the $\epsilon_j$'s are either $-1$ or $0$ and the $-1$'s occur at the coordinates
corresponding to the lines that pass through $x_i$.
Consider the subspace $X = L/pM$ of the $\FF_q$-vector space $M/pM$.
Since $X$ is isotropic with respect to the induced non-degenerate form on $M/pM$,
we have $\op{dim}_{\FF_q}(X) \leq \tfrac{1}{2}\op{dim}_{\FF_q}(M/pM) =  \tfrac{1}{2} (n+1)$.
\par
Note that $X$ is spanned by the images of the vectors $x_0, \dotsb, x_{n-1}$. Let
$\tilde{A}$ be the $n \times (n+1)$ matrix whose $i$-th row is $x_i$.
Let $A$ be the matrix obtained from $\tilde{A}$ by deleting
the first column of all $1$'s. Each row of $\tilde{A}$ add up to $-q$, so 
$A$ and $\tilde{A}$ have the same $\FF_q$-rank.
Since $(-A)$ is the incidence matrix of $\PP^2(\FF_q)$, 
theorem \ref{th-GM} implies that
$\op{rank}_{\FF_q}(A) = \frac{n+1}{2}$.
So $\op{dim}_{\FF_q}(X) = \frac{n+1}{2}$.
Thus $X$ is maximal isotropic; so lemma
\ref{l-isot} implies that $L_q$ is $p$-modular.
\end{proof}
\begin{corollary} Suppose the assumptions of theorem \ref{th-pmodular} hold. Identify  $p^{-1} L_{\cL}$ with $\cO^{1,n}$. Then
\begin{equation*}
L_q = \lbrace v \in \cO^{1,n} \colon \ip{x}{v} \equiv 0 \mod p \cO \text{\; for all \;} x \in \cP \rbrace.
\end{equation*}
\end{corollary}  
\begin{remark}
Suppose $\cO$ is the ring of integers in a quadratic imaginary number field. Then
by structure theory of modules over Dedekind domains, we know that
as a $\cO$-module $L_q \simeq \cO^n \oplus I$ where $I$ is some ideal in $\cO$.
It it not clear to me whether $L_q$ is a free $\cO$-module when $\cO$ is
not a principal ideal domain.
\end{remark}

%
%
\section{Positive definite lattices with symmetry of finite projective planes}
\label{s-definite}
\begin{lemma}
Let $(\cO, p)$ be as in \ref{intro1}.
Let $L$ be a Hermitian $\cO$-lattice such that $p L' = L$.
If $z$ is a primitive element of $L$,
then $\ip{L}{z} =  p\cO$. Since $\bar{p} \cO = p \cO$, we also have $\ip{z}{L} = p \cO$.
\label{l-dedekind}
\end{lemma}
\begin{proof} The lemma holds when $\cO = \cH$ is the ring of Hurwitz integers since
every ideal in $\cH$ is principal. Otherwise, we may assume that $\cO$ is a Dedekind domain.
Suppose $\ip{L}{ p^{-1} z} = I$ is a proper ideal in $\cO$. 
Suppose $I \cap \ZZ = s \ZZ$. 
There exists an ideal $J$ such that $I J = s \cO$.
Then, for all $ j \in J$, we have
 $\ip{L}{s^{-1} p^{-1} j z } \subseteq s^{-1} j I \subseteq s^{-1} JI  \subseteq \cO.$
It follows that $s^{-1} p^{-1} j z \in L' = p^{-1} L$, 
 so $s^{-1} j z \in L$ for all $j \in J$. Since $z$ is primitive, $s^{-1} j \in \cO$ for all $ j \in J$, so
 $J \subseteq s \cO$. But $I J = s \cO$, so $I = \cO$.
\end{proof}
\begin{lemma}
Let $(\cO, p)$ be as in \ref{intro1}.
Let $L$ be a $p$-modular Lorentzian Hermitian $\cO$-lattice.
Let $z$ be a primitive norm $0$ vector in $L$. Then $L$ splits off a hyperbolic cell containing $z$, that is,
there exists a lattice $H$ of signature $(1,1)$
containing $z$ such that $L = H \oplus \Lambda$ for a definite lattice $\Lambda \simeq z^{\bot} /z $.
Further, $\Lambda$ is also $p$-modular.
\label{l-mod}
\end{lemma}
\begin{proof}
Lemma \ref{l-dedekind} implies that there exists $f \in L$ such that $\ip{z}{f} = p$. Then $H = \cO z + \cO f$
is a hyperbolic cell. Note that $\pi_H$, given by
\begin{equation*} 
\pi_{H}(x) =  ( \bar{p}^{-1}\ip{f}{x} - \abs{p}^{-2}\abs{f}^2  \;  \ip{z}{x} ) z + p^{-1} \ip{z}{x} f
 \end{equation*}
 is the orthogonal projection of $L\otimes \op{Frac}(\cO)$ to $H \otimes \op{Frac}(\cO)$ and $\pi_{H}$ maps $L$
 into $H$.
It follows that $L  = H \oplus \Lambda$, where $\Lambda = H^{\bot}$. So $z^{\bot} = \Lambda \oplus \cO z$ and
$z^{\bot}/z \simeq \Lambda$.
\par
It remains to see that $p \Lambda' = \Lambda$. If $\lambda \in \Lambda$, then $\ip{\lambda}{L} \in p \cO$, so
$\ip{\lambda}{\Lambda} \in p \cO$, that is $\Lambda \subseteq p \Lambda'$. Suppose $\phi \in \Lambda'$.
Since $L = \Lambda \oplus H$, we can extend $\phi$ to an element of $L'$ by defining $\phi$ to be $0$
on $H$. Since $L' = p^{-1}L$, there exists $x \in L$ such that $\phi(\cdot) = \ip{p^{-1}x }{\cdot}$.
But then $\phi(\lambda) = \ip{p^{-1} (x - \pi_H(x))}{ \lambda}$ for all $\lambda \in \Lambda$
and $(x - \pi_H(x)) \in \Lambda$. 
\end{proof}
\begin{topic}{ \bf Positive definite modular lattices with $\op{PGL}(3,\FF_{q})$ symmetry: }
\label{posmod}
Let $(\cO, p,q)$ be as in \ref{intro1} and $L_{q}$ be the lattice defined in \ref{def-Lq} from this data.
Suppose $L_q$ has a primitive norm zero vector $z$ fixed by $\op{PGL}(3, \FF_{q})$.
Suppose $g \in \op{PGL}(3, \FF_{q})$ acts trivially on $z^{\bot}/z$. Since $g$ has
finite order, it must fix $z^{\bot}$ point-wise. But $g$ also point-wise fixes
the span of $w_{\cP}$ and $w_{\cL}$. So $g$ must be trivial. So the automorphism
group of $z^{\bot}/z$ contains $\op{PGL}(3, \FF_{q})$.
It follows that $z^{\bot}/z$ is a positive definite $(n-1) = q^{2} + q$ dimensional $\cO$-lattice,
whose automorphism group contains $\op{PGL}(3, \FF_{q})$. If $L$ is $p$-modular,
then lemma \ref{l-mod} implies that the positive definite lattice $z^{\bot}/z$ is also $p$-modular.
\par
For example, if $\cO = \ZZ[ \tfrac{1 + \sqrt{-3}}{2}]$ and $q = 3$, then we may take
$z = w_{\cP} + \tfrac{1}{2}(-1 + p)\bar{p} w_{\cL}$.
To find more examples of $z$ (see \ref{examples2}), we need the lemma below.
\end{topic}
\begin{lemma}
\label{l-ternary}
Let $q \equiv 3 \bmod 4$ be a rational prime, $n = q^2 + q + 1$. If $p$ is a rational prime,
write $n = p^{v_{p}(n)} m$ with $p \nmid m$.
Then the following are equivalent:
\par
(a) The integer $n$ is a norm of some element in $\QQ[\sqrt{-q}]$.
\par
(b) The ternary quadratic form 
\begin{equation} 
z^2 + q x^2 - n y^2
\label{eq-tern}
 \end{equation}
represents $0$ over $\ZZ$.
\par
(c) If $v_p(n)$ is odd for some rational prime $p$, then $p \equiv 1 \bmod 4$.
\end{lemma}
\begin{proof}
The equivalence of (a) and (b) is clear. Assume (c).
Let $(z,x,y)$ be a zero of the ternary form \eqref{eq-tern} over $\ZZ$ such that the greatest common divisor of $x,y$ and $z$
is equal to $1$. Let $p$ be a prime; $p \equiv 3 \bmod 4 $. Suppose, if possible $v_p(n) = 2 r + 1$. 
Since $q^{3} \equiv 1 \bmod p$,  
\begin{equation*} 
z^{2} + (q^{-1} x)^{2} \equiv z^{2} + q^{3} (q^{-1}x)^{2} \equiv z^{2} + q x^{2}  \equiv 0 \bmod p,
 \end{equation*}
where $q^{-1}$ denotes the inverse of $q$ modulo $p$.
The equation $Z^{2}  + X^{2} = 0$ has no nontrivial solution in $\FF_{p}$. So $p$ must divide both $z$ and $x$
and hence $p$ does not divide $y$. So $v_{p}(z^{2} + q x^{2}) = v_p(n) = 2 r +1$.
If $v_{p}(z) \neq v_{p}(x)$, then $v_{p}(z^{2} + q x^{2})  = 2 \min\lbrace v_{p}(z), v_{p}(x) \rbrace$
is even, which is not possible. So let $v_{p}(z) = v_{p}(x) = j$. 
Writing $z = p^j z_{1}$ and $x= p^j x_{1}$ we find that 
$v_{p}(z_{1}^{2} + q x_{1}^{2}) + 2j = 2r + 1$, so 
$v_{p}(z_{1}^{2} + q x_{1}^{2}) > 0$. But this again leads to
\begin{equation*} 
z_{1}^{2 } + (q^{-1} x_{1})^{2} \equiv z_{1}^{2} + q x_{1}^{2} \equiv 0 \bmod p,
 \end{equation*}
 which is a contradiction, since $p$ does not divide $z_{1}$ and $x_{1}$. Thus (b) implies (c).
  \par
 Conversely, assume (c). The ternary form represents $0$ over $\RR$, so it suffices to check that
 it represents $0$ over $\QQ_p$ for all but one rational prime $p$, that is, the local Hilbert symbols
 $(-q, n)_p = 1$. Because of the product formula (\cite{SJP:CIA} Ch. 3, theorem 3, pp. 23)
 we can omit one prime. 
 \par
 Note that $(q+1, 1, 1)$ is a nontrivial solution to $(z^2 - q x^2 - n y^2) = 0$ over $\ZZ$, so
 $(q, n)_p = 1$ for all prime $p$. So, for $p \neq 2$, using \cite{SJP:CIA} Ch. 3, theorem 1, pp. 20, we get
  \begin{equation*}
 (-q, n)_p = (-1,n)_p (q,n)_p = (-1,n)_p = \Bigl( \frac{-1}{p} \Bigr)^{v_p(n)}.
 \end{equation*}
 But our assumption states that if $v_p(n)$ is odd, then $p\equiv 1 \bmod 4$, so $-1$ is a quadratic
 residue modulo $p$.
\end{proof}
\begin{remark} Suppose $q \neq 3$ is a prime such that 
the conditions of lemma \ref{l-ternary} are satisfied.
If $q$ is of the form $3k + 1$, then
$n = 9k^2 + 9 k + 3$, so $v_3(n) = 1$, which is not possible. So if $q \neq 3$, then we must have $q \equiv -1 \bmod 12$.
The first few primes $q$ satisfying the conditions in \ref{l-ternary}, (d) are $q = 3, 47, 59, 71, 131$.
For example, if $q = 3$, then $(z,x,y) = (1,2,1)$ is a solution to \eqref{eq-tern}. 
If $q = 47$, then $(z,x,y) = (47, 27, 4)$ is a solution.
General conjectures like Schinzel's hypothesis H or Bateman-Horn conjecture
imply that there are infinitely such primes.
\end{remark}
\begin{example}
\label{examples2}
(1) Let $(\cO, p, q)$ be as in \ref{examples1}(1). One verifies that
the group $\op{PGL}(3, \FF_q)$ fixes a two dimensional subspace of $L_{q} \otimes_{\cO} \QQ[\sqrt{-q}]$
spanned by $w_{\cP}$ and $w_{\cL}$. 
So $L_q$ contains a norm zero vector fixed by $\op{PGL}(3, \FF_q)$
if and only if $(w_{\cP} + c w_{\cL})$ has norm zero for some
$c \in \QQ[\sqrt{-q}]$. Using equation \eqref{eq-ipwpwl}, one verifies that
\begin{equation*}
\abs{w_{\cP} + c w_{\cL}  }^2  = \abs{c p + q + 1}^2 - (q^2 + q + 1).
\end{equation*}
So $\abs{w_{\cP} + c   w_{\cL}}^2 = 0$ if and only if  $(q^2 + q + 1)$ is a norm of some element of $ \QQ[\sqrt{-q}]$.
\par
Suppose $q$ is such that the conditions in lemma \ref{l-ternary} hold.
Let $z$ be a primitive norm zero vector in $L_q$ fixed by $\op{PGL}(3, \FF_q)$.
By lemma \ref{l-dedekind}, there exists a lattice vector $f$ such that
$\ip{z}{f} = p$. So we can take $ H = \cO z + \cO f$. 
Writing $L_q = \Lambda_q \oplus H$ as in \ref{l-mod}, we get a $p$-modular Hermitian
lattice $\Lambda_q$ defined over $\ZZ[(1 + p)/2]$, whose automorphism group contains $\op{PGL}(3, \FF_q)$.
For $q  = 3$ we find that $\Lambda_q$ is the Leech lattice defined over Eisenstein integers.
\par
(2) Let $\cO = \cH$ be the ring of Hurwitz integers, $p = (1 - i)$ and $q = 2$. Let $L_2^{\cH}$ be the lattice
obtained from this data. 
(Of course in this case one has to be careful to phrase everything in terms of right modules or left modules.)
The reflection group of $L_2^{\cH}$ was studied in \cite{TB:QL} where we always considered right $\cH$-modules.
One checks that $z = w_{\cP} + w_{\cL} \bar{p} (-1 + i + j + k)/2 $ is a primitive null vector in $L_2^{\cH}$.
So we can write $L_2^{\cH} = \Lambda \oplus H$, so that
$\Lambda \simeq z^{\bot}/z$ and $H$ is a hyperbolic cell. The lattice $\Lambda$ is a quaternionic form of
Leech lattice defined of Hurwitz integers.
\end{example}
\begin{topic}{\bf  Even unimodular positive definite $\ZZ$-lattices with $\op{PGL}(3, \FF_q)$ symmetry:}
\label{unimod}
Let $q$ be a rational prime; $q\equiv 3 \bmod 4$. Suppose $q$ satisfies the
conditions in lemma \ref{l-ternary}.
Let $\Lambda_q$ be the definite Hermitian $\cO$-lattice from \ref{examples2}(1).
Let $M_q$ be the underlying $\ZZ$-module of $\Lambda_q$ with the 
integral bilinear form 
\begin{equation*}  
(x,y) = -2q^{-1} \op{Re}\ip{x}{y}.
\end{equation*}
Then $M_q$ is a positive definite, even, unimodular $\ZZ$-lattice
of dimension $2q(q+1)$ such that $\op{Aut}(M_q) \supseteq \op{PGL}(3, \FF_q)$.
If $q =3 $, then $M_q$ is the Leech lattice. 
\begin{proof}[proof that $M_q$ is unimodular:]
Identify the vector spaces $\Lambda_q \otimes_{\cO} \QQ(\sqrt{-q})$ and $M_q \otimes_{\ZZ} \QQ$.
All the lattices in question can be identified inside this vector space.
Suppose $\mu \in M_q'$. Let $\ip{\mu}{y} = (u  + p v)/2$ with $u, v \in \RR$.
Since $( \mu, y) = -2 q^{-1}\op{Re}\ip{\mu}{y} \in \ZZ$, 
we have $u \in q \ZZ$. Also  
\begin{equation*} 
(\mu, \tfrac{(1 + p)}{2} y )  = -2 q^{-1} \op{Re} \ip{\mu}{ \tfrac{(1 + p)}{2} y } = (q v - u)/2 q \in \ZZ.
 \end{equation*}
 So $v \in q^{-1} u + 2 \ZZ$. It follows that $v \in \ZZ$ and $u \equiv v \bmod 2$. So 
 $\ip{\mu}{y}  \in p \cO$. So $p^{-1} \mu \in {\Lambda}_q' = p^{-1} {\Lambda}_q$, so $\mu \in {\Lambda}_q$,
 that is $\mu \in M_q$.
\end{proof}
\end{topic}
%
%
\section{Arithmetic of finite fields and finite projective geometry}
\label{s-inc}
\begin{topic}
In this section we shall explain how to write down the incidence matrix
of $\PP^2(\FF_q)$ from the observation:
$\PP^2(\FF_q) \simeq \FF^*_{q^3}/\FF^*_q.$
This lets us write down the vectors $x_i$ explicitly in line coordinates (see \ref{t-linec}).
This also provides some of the ingredients of the proof of \ref{th-GM}.
The results in this section are contained in \cite{SJ:FP}. But since that paper goes
back to 1934, obviously we can make the argument quicker. We include it for the sake of completeness.
\end{topic}
\begin{theorem}[\cite{SJ:FP}]
There is an element $T$ of $\op{PGL}(r+1, \FF_q)$ that acts transitively on $\PP^r(\FF_q)$.
The points and hyperplanes in $\PP^r(\FF_q)$ can be labeled as $\lbrace x_i \colon i \in \ZZ/n \ZZ \rbrace$
and  $\lbrace l_i \colon i \in \ZZ/n \ZZ\rbrace$ such that $T l_{i} = l_{i+1}$ and $T x_i = x_{i+1}$
for all $i \in \ZZ/n \ZZ$. So if
\begin{equation} 
l_0 = \lbrace x_{d_0}, x_{d_1}, \dotsb, x_{d_q} \rbrace,
\label{eq-l_0}
 \end{equation}
 Then
 \begin{equation*} 
l_{j} =\lbrace x_{d_0 + j}, x_{d_1 + j}, \dotsb, x_{d_q+j} \rbrace; \text{\; for \;} j = 0, 1, 2, \dotsb, {n-1}.
 \end{equation*}
\label{th-trans}
\end{theorem}
\begin{proof} We shall write the proof for $r = 2$. The general proof
works with obvious modifications.
Pick a generator $\lambda$ for the multiplicative group $\FF^*_{q^3}$.
Identify $\FF^*_q$ inside $\FF^*_{q^{3}}$ as
$$\FF^*_q  = \lbrace 1, \lambda^n, \lambda^{2n} \dotsb, \lambda^{(q - 2)n} \rbrace, \text{\; where \;} n = (q^3 -1)/(q - 1) = 1 + q + q^2.$$
So $\PP^2(\FF_q)$ obtains a structure of an cyclic group; we identify $\PP^2(\FF_q) = \FF^*_{q^{3}}/\FF^*_q$. We shall write 
\begin{equation} 
\PP^2(\FF_q) = \lbrace x_0, x_1, \dotsb, x_{n-1} \rbrace   \text{\; where \;} x_i = \lambda^i \FF_q^*.
\label{eq-defxi}
 \end{equation}
 Note that $x \mapsto \lambda x$ is a $\FF_q$-linear endomorphism of $\FF_{q^{3}}$, which is transitive on $\FF^*_{q^3}$.
 This map defines an element $T$ of $\op{PGL}(3, \FF_q)$ that is transitive on $\PP^2(\FF_q)$. 
\par
We identify the dual vector space $(\FF_{q^3})'$ with $\FF_{q^3}$
using the trace form $(x, y) \mapsto \Tr(xy)$, where 
$\Tr: \FF_{q^3} \to \FF_q$ is the trace map: $\Tr(x) = x + x^q + x^{q^2}$.
This induces an isomorphism $\sigma$ from the points of 
$\PP^2(\FF_q)$ to the lines of $\PP^2(\FF_q)$. 
Let  $l_i = \sigma(x_{n-i})$, for $i = 0, 1, 2, \dotsb, n-1$.
So the points on the line $l_{i}$ correspond to the set of $x$'s such that
$\Tr(x_{n-i} x) = 0$.
Observe that $x \in l_{i}$ if and only if $\Tr(\lambda^{n-i} x) = 0$, if and only if 
$\Tr(\lambda^{-i}x)= 0$, if and only if $\lambda x \in l_{i+1}$.
So $T l_{i} = l_{i+1}$.
\end{proof}
\begin{definition}
A subset $\Delta \subseteq \ZZ/n \ZZ$ is called a {\it perfect difference set} if given any non-zero
$r \in \ZZ/n \ZZ$, there exists a unique ordered pair $(d, d')$
from $\Delta$ such that $(d - d') = r$.
\end{definition}
\begin{lemma} Assume the setup of lemma \ref{th-trans} with $r = 2$. 
Let $\Delta  = \lbrace d_0, d_1, \dotsb, d_q \rbrace$
be the set of indices corresponding to the points on the line $l_0$ of $\PP^2(\FF_q)$. Then 
$\Delta$ is a perfect difference set.
 \label{l-diff}
\end{lemma}
\begin{proof}
 If $r = d_j - d_i = d_k - d_j \bmod n$ for three distinct elements $d_i, d_j, d_k \in \Delta$
then $x_{d_j}$ and $x_{d_k}$ belongs to both $l_0$ and $l_r$ which is impossible since
two lines intersect at one point. So the differences $(d_i - d_j)$ are all distinct modulo $n$
for all pair of distinct elements $(d_i, d_j)$ from $\Delta$. Since there are $q(q+1) = n -1$ such pairs,
the lemma follows.
\end{proof}
\begin{topic}{\bf Cyclic difference set codes and the incidence matrix of the projective plane: }
 Let $q = l^d$ be a power of a rational prime $l$.
Let $A$ be the $n \times n$ incidence matrix of $\PP^2(\FF_q)$. 
The $(i,j)$-th entry of $A$ is $1$ if $x_j \in l_i$ and is zero otherwise.
\par
Identify $\FF_q^n$ with $R = \FF_q[t]/\langle t^n - 1 \rangle$.
Linear cyclic codes of length $n$ over $\FF_q$ are
principal ideals in $R$. Consider the polynomial $\theta(t) = \sum_{d \in \Delta} t^d,$ 
where $\Delta$ is as in lemma \ref{l-diff}.
Theorem \ref{th-trans} implies that the $i$-th row of the incidence matrix $A$
corresponds to the polynomial $t^i \theta(t)$. 
So the span of the rows of the incidence matrix of $\PP^2(\FF_q)$ is isomorphic to
the principal ideal $R \theta(t)$ and theorem \ref{th-GM} is equivalent to the equality
$\op{dim}_{\FF_q}( R \theta(t) ) = \binom{ l + 1}{2}^d + 1.$
This is proved in \cite{GM:DCC}.
\end{topic}
\begin{example} Let $\lambda$ be a generator of $\FF_{27}^*$ such that 
$\lambda^3 = \lambda + 1$. Then the $13$ points of $\PP^2(\FF_3) = \FF^*_{27}/\FF^*_3$
are the cosets of $x_0 = 1, x_1 = \lambda, \dotsb, x_{12} = \lambda^{12}$. The $13$ lines are 
$l_0, \dotsb, l_{12}$ where the points on $l_i$ are the points $x$ such that $\Tr(\lambda^{-i} x) =0$.
So the points on $l_0$ correspond to the $x$'s  such that 
$\Tr(x) = x + x^3 + x^9 = 0$. 
\par
Clearly $\Tr(1) = 0$.  We have $\Tr(\lambda)= 0$, 
since the trace is the coefficient of the degree two term in the minimal polynomial
of $\lambda$. 
Since $\lambda^3$ and $\lambda^9$ are the Galois conjugates of $\lambda$, we have
$\Tr(\lambda^3) = \Tr(\lambda^9) = 0$. It follows that
$l_0 = \lbrace x_0, x_1, x_3, x_9 \rbrace$.
Now we can write down the the incidence graph of $\PP^2(\FF_3)$ using lemma \ref{th-trans}, and hence
write down the line coordinates for the vectors $x_0, \dotsb, x_{12}$ in $L_3$.
\end{example}
%
%
%
%
%
%
%
%
%
%

%
%
%
%
\end{document}